\documentclass[11pt]{amsart}
\usepackage[utf8]{inputenc}
\usepackage[a4paper,margin=3cm]{geometry}
\usepackage[pdftex,pdfborder={0 0 0},colorlinks,pagebackref]{hyperref} 
\usepackage{latexsym,lmodern,graphicx,enumitem,amssymb,dsfont}

\newtheorem{thm}[equation]{Theorem}
\newtheorem{lem}[equation]{Lemma}
\newtheorem{prop}[equation]{Proposition}
\newtheorem{cor}[equation]{Corollary}

\newtheorem{rem}[equation]{Remark}
\numberwithin{equation}{section}

\DeclareRobustCommand{\rchi}{{\mathpalette\irchi\relax}}
\newcommand{\irchi}[2]{\raisebox{\depth}{$#1\chi$}}

\newcommand{\ra}{\rightarrow}

\newcommand{\R}{\mathbb{R}}
\newcommand{\RL}{\mathbb{R}}

\begin{document}
\title[Deviation inequalities for convex functions]{Deviation inequalities for convex functions motivated by the Talagrand  conjecture}
\author{Nathael Gozlan, Mokshay Madiman, Cyril Roberto, Paul-Marie Samson}

\thanks{Supported by the grants ANR 2011 BS01 007 01, ANR 10 LABX-58, ANR11-LBX-0023-01}

\address{Universit\'e Paris Descartes - MAP 5 (UMR CNRS 8145), 45 rue des Saints-Pères 75270 
Paris cedex 6, France.}
\address{University of Delaware, Department of Mathematical Sciences, 501 Ewing Hall, Newark DE 19716, USA.}
\address{Universit\'e Paris Ouest Nanterre La D\'efense - Modal'X, 200 avenue de la R\'epublique 92000 Nanterre, France}
\address{Universit\'e Paris Est Marne la Vall\'ee - Laboratoire d'Analyse et de Math\'e\-matiques Appliqu\'ees (UMR CNRS 8050), 5 bd Descartes, 77454 Marne la Vall\'ee Cedex 2, France}

\email{natael.gozlan@parisdescartes.fr, madiman@udel.edu, croberto@math.cnrs.fr, paul-marie.samson@u-pem.fr}

\date{\today}

\maketitle

\begin{abstract}
Motivated by Talagrand's conjecture on regularization properties of the natural semigroup on the Boolean hypercube,
and in particular its continuous analogue involving regularization properties of the Ornstein-Uhlenbeck semigroup acting on integrable functions,
we explore deviation inequalities for log-semiconvex functions under Gaussian measure.
\end{abstract}


\section{Introduction}
In the late eighties, Talagrand conjectured that the ``convolution by a biased coin'', on the hypercube $\{-1,1\}^n$, satisfies some refined hypercontractivity property. We refer to Problems 1 and 2 in \cite{Tal89:1} for precise statements. 
A continuous version of Talagrand's conjecture for the Ornstein-Uhlenbeck operator has recently attracted some attention \cite{BBBOW13, EL14:2, Leh16}; 
in particular, it was resolved by  \cite{EL14:2, Leh16} by first proving a deviation inequality for log-semiconvex functions above their means under Gaussian measure.
In this paper, we discuss a simpler approach to proving this deviation inequality for the special case of log-convex functions (which is already of interest).

\medskip

Let us start by presenting the continuous version of Talagrand's conjecture and the history of its resolution.
Denote by $\gamma_n$ the standard Gaussian (probability) measure in dimension $n$, with density 
\[
x \mapsto (2 \pi)^{-n/2} \exp \left\{ -\frac{|x|^2}{2} \right\}
\]
(where $|x|$ denotes the standard Euclidean norm of $x\in \R^n$) and, for $ p \geq 1$, by $\mathbb{L}^p(\gamma_n)$ the set of measurable functions $f \colon \mathbb{R}^n \to \mathbb{R}$ such that $|f|^p$ is integrable with respect to $\gamma_n$. Then, given $g \in \mathbb{L}^1(\gamma_n)$, the Ornstein-Ulhenbeck semi-group is defined as
\begin{equation} \label{eq:ou}
P_tg(x) := \int g\left(e^{-t}x + \sqrt{1-e^{-2t}}y\right)\,d\gamma_n(y) \qquad \qquad x \in \mathbb{R}^n , \; t \geq 0.
\end{equation}

It is well known that the family $(P_t)_{t \geq 0}$ enjoys the so-called hypercontractivity property \cite{Nel66,Nel73,Gro75} which asserts that, for any $p>1$, any $t>0$ and $q \leq 1 + (p-1)e^{2t}$, $P_t g$ is more regular than $g$ in the sense that,
if $g \in \mathbb{L}^p(\gamma_n)$ then $P_tg \in \mathbb{L}^q(\gamma_n)$ and moreover
\[
\| P_t g\|_q \leq \|g\|_p. 
\]
However this property is empty when one only assumes that $g \in \mathbb{L}^1(\gamma_n)$. A natural question is therefore to ask if the semi-group has anyway some regularization effect also in this case. Given $g : \mathbb{R}^n \to \mathbb{R}$ non-negative  with $\int g \,d\gamma_n =1$, by Markov's inequality and the fact that $\int P_s g\,d\gamma_n=1$ we have
\[
\gamma_n(\{P_sg \geq t \}) \leq \frac{1}{t} \qquad \qquad \forall t>0 . 
\]
The continuous version of Talagrand's conjecture (adapted from \cite[Problems 1 and 2]{Tal89:1}) states that as soon as $s>0$,
\[
\lim_{t \to \infty} \sup_{g \geq 0, \int g\,d\gamma_n=1}t \gamma_n(\{P_s g \geq t \}) = 0.
\] 
The most recent paper dealing with this conjecture is due to Lehec \cite{Leh16} who proved that, for any $s >0$ there exists a constant $\alpha_s \in (0, \infty)$ (depending only on $s$ and not on the dimension $n$) such that for any non-negative function $g : \mathbb{R}^n \to \mathbb{R}^+$ with $\int g\, d\gamma_n =1$,
\begin{equation} \label{Leh16}
\gamma_n(\{P_s g \geq t \}) \leq \frac{\alpha_s}{t \sqrt{\log t}} \qquad \qquad \forall t>1 
\end{equation}
and this bound is optimal in the sense that the factor $\sqrt{\log t}$ cannot be improved. In the first paper dealing with this question \cite{BBBOW13}, Ball, Barthe, Bednorz, Oleszkiewicz and Wolff already obtained a similar bound but with a constant $\alpha_s$ depending heavily on the dimension $n$ plus some extra $\log \log t$ factor in the numerator. Later Eldan and Lee \cite{EL14:2} proved that the above bound holds with a constant $\alpha_s$ independent on $n$ but again with the extra $\log \log t$ factor in the numerator. Finally the conjecture was fully proved by Lehec removing the $\log \log t$ factor \cite{Leh16} and giving an explicit bound on $\alpha_s$, namely that $\alpha_s := \alpha \max(1,\frac{1}{2s})$ for some numerical constant $\alpha$.

\medskip

In both Eldan-Lee and Lehec's papers, the two key ingredients are the following: 
\begin{itemize}
\item[(1)] for any $s>0$, the Ornstein-Uhlenbeck semi-group satisfies, for all non-negative function $g \in \mathbb{L}^1(\gamma_n)$,
\[
\mathrm{Hess}\,(\log P_s g) \geq -\frac{1}{2s} \mathrm{Id},
\]
where $\mathrm{Hess}$ denotes the Hessian matrix and $\mathrm{Id}$ the identity matrix of $\mathbb{R}^n$.
This is a somehow standard property easy to prove thanks to the kernel representation \eqref{eq:ou};
\item[(2)] for any positive function $g$ with $\mathrm{Hess}\,(\log g) \geq -\beta \mathrm{Id}$, for some $\beta>0$, and $\int g \,d\gamma_n =1$, it holds
\[
\gamma_n(\{g \geq t \}) \leq \frac{C_\beta}{t \sqrt{\log t}} \qquad \qquad \forall t>1,
\]
with $C_\beta = \alpha \max(1, \beta)$.
\end{itemize}

\medskip

It will be more convenient to deal with $g=e^f$ in the sequel so we move to this setting now. The last inequality can be reformulated as follows: 
for any $f : \mathbb{R}^n \to \mathbb{R}$ with $\int e^f\, d\gamma_n =1$ and $\mathrm{Hess}\,(f) \geq -\beta \mathrm{Id}$, it holds
\begin{equation} \label{eq:start}
\gamma_n \left(\left\{ f \geq  t \right\} \right) \leq C_\beta \frac{e^{-t}}{\sqrt{t}} \qquad \qquad \forall t >0 .
\end{equation}

\medskip

We now describe the two main contributions of this note
(which
were independently obtained by Ramon van Handel). 
First, as a warm up, we give in Section~\ref{sec:tal-1} a short proof of \eqref{eq:start} in dimension 1. The main argument of this 
proof is that due to the semi-convexity of $f$, the condition $(2\pi)^{-1/2}\int e^{f-\frac{1}{2}|x|^2}\,d\gamma =1$ 
implies a \emph{pointwise} comparison between $f$ and the function $|x|^2/2$, which then can be turned into a tail comparison. 

\medskip

Then, in dimension $n$, we give in Section~\ref{sec:logcvx} a sharp version of the upper bound \eqref{eq:start} for convex functions.
Our main result states:

\begin{thm}\label{main-result}
Suppose that $f: \R^n \to \R$ is a convex function such that $\int e^f\,d\gamma_n=1$, then
\begin{equation}\label{eq:sharp}
\gamma_n (f \geq t) \leq \overline{\Phi}(\sqrt{2t}),\qquad \forall t\geq0,
\end{equation}
where $\overline{\Phi}(t) = \frac{1}{\sqrt{2\pi}} \int_t^{+\infty} e^{-u^2/2}\,du$, $t\in \R.$
\end{thm}

Let us make a few comments on this result. First, using the following classical bound (which is asymptotically optimal)
\begin{equation}\label{eq:boundPhibar}
\overline{\Phi}(s) = \frac{1}{\sqrt{2\pi}}\int_s^\infty e^{-x^2/2}\, dx 
\leq  
\frac{1}{\sqrt{2\pi}}\int_s^\infty \frac{x}{s} e^{-x^2/2}\,dx
= 
 \frac{e^{-s^2/2}}{\sqrt{2\pi}s}, \qquad \forall s>0,
\end{equation}
one immediately recovers \eqref{eq:start} with the constant $C_0' = 1/(2\sqrt{\pi}).$ Furthermore, the bound \eqref{eq:sharp} is sharp. 
Indeed, for a given value of $t\geq 0$, Inequality \eqref{eq:sharp} becomes an equality for the function 
\[
f_t(x) = \sqrt{2t}x_1 -t,\qquad x = (x_1,\ldots,x_n) \in \R^n.
\]
Finally, since the Ornstein-Uhlenbeck semigroup preserves log-convexity (this follows from the fact that any positive combination of 
log-convex functions remains log-convex, see e.g \cite{MOA11:book} p. 649), Theorem \ref{main-result} immediately implies the following corollary.

\begin{cor}\label{cor:lc-reg}
Let $g$ be a log-convex function such that $\int g\,d\gamma_n=1$, then for any $s\geq 0$,
\[
\gamma_n(P_sg \geq t) \leq \overline{\Phi}(\sqrt{2\log(t)}),\qquad \forall t\geq1.
\]
\end{cor}

In the special case when $g$ is log-convex, Corollary~\ref{cor:lc-reg} is a sharp improvement of Lehec's result \eqref{Leh16}. 
Note that for log-convex $g$, the constant $\alpha_s$ can be taken independent of $s$ unlike in \eqref{Leh16}, but this already
followed from Lehec's inequality \eqref{eq:start} combined with the preservation of log-convexity by the Ornstein-Uhlenbeck semigroup.

Another consequence of Theorem~\ref{main-result} is that a deviation inequality for structured functions also follows
for other measures that can be obtained by ``nice'' pushforwards of Gaussian measure. 
Indeed, observe that for any coordinate-wise non-decreasing, convex function $f$ on $\RL^n$,
and any convex functions $g_1,\ldots, g_n:\RL^N\ra\RL$,
the composition $f(g_1(x), \ldots, g_n(x))$
is convex on $\RL^{N}$. Hence we immediately have the following corollary.

\begin{cor}\label{cor:other-meas}
For a standard Gaussian random vector $Z$ in $\RL^N$, let the probability measure $\mu$ on $\RL^n$
be the joint distribution of $(g_1(Z), \ldots, g_n(Z))$, where $g_1,\ldots,g_n : \RL^N\ra\RL$ are convex functions.
Suppose that $f: \R^n \to \R$ is a coordinate-wise non-decreasing, convex function such that $\int_{\RL^n} e^f\,d\mu=1$. Then
\begin{equation}\label{eq:sharp}
\mu (f \geq t) \leq \overline{\Phi}(\sqrt{2t}),\qquad \forall t\geq0,
\end{equation}
\end{cor}

For example, consider the exponential distribution, whose density is $e^{-x}$ on $\RL_+=(0,\infty)$ and which can be realized
as $\frac{Z_1^2+Z_2^2}{2}$ with $Z_1, Z_2$ i.i.d. standard Gaussian. Clearly a product of exponential distributions on the line
is an instance covered by Corollary~\ref{cor:other-meas}, since we can take $N=2n$
and $g_i(x)=\frac{x_i^2+x_{i+1}^2}{2}$. More generally, Corollary~\ref{cor:other-meas} applies to 
a product of $\chi^2$ distributions with arbitrary degrees of freedom, and also to some cases
with correlation (consider for example $N=3, g_1(x)=\frac{x_1^2+x_{2}^2}{2}$ and $g_2(x)=\frac{x_2^2+x_{3}^2}{2}$).


\medskip

The proof of Theorem \ref{main-result} is given in Section~\ref{sec:logcvx}. It relies on the Ehrhard inequality, which we recall now: according to \cite[Theorem 3.2]{Ehr83}, if $A,B \subset \R^n$ are two convex sets, then 
\begin{equation}\label{eq:Ehr83}
\Phi^{-1} (\gamma_n(\lambda A + (1-\lambda) B)) \geq \lambda \Phi^{-1} (\gamma_n(A)) + (1-\lambda) \Phi^{-1} (\gamma_n(B)),\qquad \forall \lambda \in [0,1],
\end{equation}
where $\lambda A + (1-\lambda) B := \{\lambda a + (1-\lambda)b : a \in A, b \in B\}$ denotes the usual Minkowski sum and $\Phi^{-1}$ is the inverse of the cumulative distribution function $\Phi$ of $\gamma_1$:
\begin{equation}\label{eq:Phi}
\Phi(t) = \frac{1}{\sqrt{2\pi}} \int_{-\infty}^t e^{-u^2/2}\,du,\qquad t \in \R.
\end{equation}
After Ehrhard's pioneer work, Inequality \eqref{eq:Ehr83} was shown to be true if only one set is assumed to be convex by Lata{\l }a \cite{Lat96} and finally to arbitrary measurable sets by Borell \cite{Bor03}. See also \cite{BH09,Han16} and the references therein for recent developments on this inequality. Inequality \eqref{eq:Ehr83} (for arbitrary sets $A,B$) is a very strong statement in the hierarchy of Gaussian geometric and functional inequalities. For instance, it gives back the celebrated Gaussian isoperimetric result of Sudakov-Tsirelson \cite{ST74} and Borell \cite{Bor75b}. Another elegant consequence of \eqref{eq:Ehr83} due to Kwapie{\'n} is that if $f$ is a convex function on $\R^n$ which is integrable with respect to $\gamma_n$, then the median of $f$ is always less than or equal to the mean of $f$ under $\gamma_n$.
The key ingredient in Kwapie{\'n}'s proof is the observation that the function 
\[
\alpha(t) = \Phi^{-1}(\gamma_n(f \leq t)),\qquad t \in \R
\]
is \emph{concave} over $\R$; this observation (already made in Ehrhard's original paper) also plays a key role in our proof of Theorem \ref{main-result}.

After the completion of this work, we learned that Paouris and Valettas \cite{PV16} developed in a recent paper similar ideas to derive from \eqref{eq:Ehr83} deviation inequalities for convex functions \emph{under} their mean.

In Section~\ref{sec:logsemicvx}, we give a second proof of Theorem \ref{main-result}, and also discuss
(following an observation of R. van Handel) the difficulty of its extension to the log-semiconvex case.



\medskip
\noindent{\bf Acknowledgement.} The results of this note 
were independently obtained by Ramon van Handel 
a few months before us, as we learnt after a version of this note was circulated.
Although he chose not to publish them, these observations should be considered as due to him.
We are also grateful to him for numerous comments on earlier drafts of this note.

\section{The Continuous Talagrand Conjecture in dimension 1}
\label{sec:tal-1}

In the next lemma we take advantage of the semi-convexity property $\mathrm{Hess}\,(f) \geq -\beta \mathrm{Id}$ to derive information on $f$. 
More precisely we may compare $f$ to $x \mapsto |x|^2/2$. The result holds in any dimension, and we give two proofs for completeness.

\begin{lem} \label{foot}
Let $f \colon \mathbb{R}^n \to \mathbb{R}$ and $\beta \geq 0$ be such that $\int e^f\,d\gamma_n =1$, $f$ is smooth and $\mathrm{Hess}\,(f) \geq -\beta \mathrm{Id}$.
Then,
\[
f(x) \leq \frac{n}{2}  \ln(1+\beta) +  \frac{1}{2}|x|^2 , \qquad \forall x \in \mathbb{R}^n .
\] 
\end{lem}

\proof[First proof of Lemma \ref{foot}]
Let $h(x)=f(x)+ \frac{\beta}{2}|x|^2$. By assumption on $f$, the function $h$ is convex on $\R^n$ and hence 
\[
h(x)= \sup_{t \in \mathbb{R}^n} \left\{ \langle x,t\rangle - h^*(t) \right\},\qquad \forall x \in \R^n,
\] 
where 
\[
h^*(t):= \sup_{x \in \mathbb{R}^n} \left\{ \langle t,x\rangle - h(x) \right\},\qquad t \in \R^n
\]
is the Legendre transform of $h$. 
Now, we have for all $t \in \mathbb{R}^n$
\begin{align*}
1 
& = 
\int e^f\, d\gamma_n = \int \exp\left\{ h(x) - \frac{\beta}{2}|x|^2 \right\}\, d\gamma_n(x) \\ 
& \geq 
(2 \pi)^{-n/2} e^{-h^*(t)} \int \exp \left\{ \langle x,t\rangle - \frac{1+\beta}{2} |x|^2 \right\}\, dx \\
& =
(1+\beta)^{-n/2} \exp \left\{ -h^*(t) + \frac{1}{2(1+\beta)}|t|^2 \right\} .  
\end{align*}
Therefore, for all $ t \in \mathbb{R}^n$ it holds
\[
h^*(t) \geq - \frac{n}{2} \ln(1+\beta) + \frac{1}{2(1+\beta)}|t|^2 .
\]
In turn
\begin{align*}
h(x) 
& = 
\sup_t \left\{ \langle x,t\rangle - h^*(t) \right\} 
\leq 
\frac{n}{2} \ln(1+\beta) + \sup_t \left\{ \langle x,t\rangle - \frac{1}{2(1+\beta)}|t|^2 \right\} \\
& =
\frac{1}{2} \left(n \ln(1+\beta) + (1+\beta) |x|^2 \right) 
\end{align*}
which leads to the desired conclusion.
\endproof

\proof[Second proof of Lemma \ref{foot}]
Define $\tilde{h}(x) = h(x) + \frac{\beta}{2}|x|^2$, $x \in \R^n$ and let $\gamma_{n,\beta}$ be the gaussian measure $\mathcal{N}(0, \frac{1}{1+\beta}I)$, then it holds 
\[
1=\int e^{h(x)}\,d\gamma_n(x) = (1+\beta)^{-n/2}\int e^{\tilde{h}(x)}\,d\gamma_{n,\beta}(x)
\]
For all $a \in \R^n$, the change of variable formula then gives 
\[
1= (1+\beta)^{-n/2}e^{-\frac{(1+\beta)}{2}|a|^2}\int e^{\tilde{h}(y+a)-(1+\beta)y\cdot a}\,d\gamma_{n,\beta}(dy).
\]
The function $y\mapsto \tilde{h}(y+a)-(1+\beta)y\cdot a$ is convex and the function $x \mapsto e^x$ is convex and increasing so the function $y\mapsto \exp\left(\tilde{h}(y+a)-(1+\beta)y\cdot a\right)$ is also convex. So applying Jensen inequality yields to
\begin{align*}
1 &\geq (1+\beta)^{-n/2}e^{-\frac{(1+\beta)}{2}|a|^2} \exp \left(\tilde{h}\left(a+\int y\,d\gamma_{n,\beta}(y)\right) - (1+\beta)\int y\cdot a\,d\gamma_{n,\beta}(y)\right)\\
& = e^{-\frac{(1+\beta)}{2}|a|^2+ \tilde{h}(a)}
\end{align*}
and so $h(a) \leq |a|^2/2 + \frac{n}{2}\log(1+\beta).$
\endproof

\begin{rem}
The $\beta=0$ case of Lemma~\ref{foot} (i.e., for convex functions $f$, which is the essential case) is contained in Graczyk et al. \cite[Lemma 3.7]{GKLZ08}
(curiously it does not appear in the published version \cite{GKL10} of the paper), and in fact was proved in the more general setting
of subharmonic functions. The second proof given above is theirs and works for the more general setting. Also note that neither proof
requires smoothness of $f$, which however is sufficient for our purposes.
\end{rem}

In principle, one would hope to already get some deviation bound from the above lemma. More precisely, given $f$ as in Lemma \ref{foot}, we have
$$
\gamma_n \left( \left\{ f \geq t \right\}\right) 
\leq 
\gamma_n \left( \left\{ |x|^2 \geq 2t - n \ln(1+\beta) \right\}\right) ,
$$
thanks to Lemma \ref{foot},
and we are left with a tail estimate for a $\rchi^2$ distribution with $n$ degrees of freedom.
In dimension $n=1$, the tail of the $\rchi^2$ distribution behaves like $e^{-t}/\sqrt{t}$. Therefore, the above simple argument already gives back 
the estimate \eqref{eq:start} and thus provides a quick proof of the continuous Talagrand's conjecture for $n=1$, moreover with clean dependence
on $\beta$, as detailed below.

\begin{thm} \label{th:dim1}
If $f \colon \mathbb{R} \to \mathbb{R}$ is smooth and $\beta \geq 0$ are such that $\int e^f\, d\gamma_1 =1$ and $f'' \geq -\beta $ pointwise, then
\[
\gamma_1 \left(\left\{ f \geq  t \right\} \right) \leq \frac{1+\beta}{\sqrt{2}} \frac{e^{-t}}{\sqrt{t}} \qquad \qquad \forall t \geq1 .
\]
\end{thm}

\begin{proof}
Assume first that $t \geq (1+\beta)\ln(1+\beta)/(2\beta)$. Using Inequality \eqref{eq:boundPhibar},
we get from Lemma \ref{foot}
\begin{align*}
\gamma_1 \left( \left\{ f \geq t \right\}\right) 
& \leq 
\gamma_1 \left( \left\{ |x| \geq \sqrt{2t - \ln(1+\beta)} \right\}\right)  \\
&\leq 
2(2 \pi)^{-1/2} \frac{\exp\left\{ -t + \frac{1}{2} \ln(1+\beta)\right\}}{\sqrt{2t-\ln(1+\beta)}} \\
&= 
\sqrt{\frac{1+\beta}{\pi}} \frac{e^{-t}}{\sqrt{t}} \frac{1}{\sqrt{1-(\ln(1+\beta)/(2t))}} \\
& \leq  
\sqrt{\frac{1+\beta}{\pi}} \frac{e^{-t}}{\sqrt{t}} \frac{1}{\sqrt{1-(\beta/(1+\beta))}}
= \frac{1+\beta}{\sqrt{\pi}} \frac{e^{-t}}{\sqrt{t}} .
\end{align*}
Now assume that $t \leq (1+\beta)\ln(1+\beta)/(2\beta)$. 
Thanks to Markov's inequality, we have
\begin{align*}
\gamma_1 \left(\left\{ f > t \right\} \right) 
 \leq 
e^{-t} 
\leq 
  \sqrt{(1+\beta)\ln(1+\beta) /(2\beta)} \frac{e^{-t}}{\sqrt{t}} 
 \leq
\frac{\sqrt{1+\beta}}{\sqrt{2}} \frac{e^{-t}}{\sqrt{t}}
\leq
\frac{1+\beta}{\sqrt 2} \frac{e^{-t}}{\sqrt{t}}
\end{align*}
where, in the third inequality, we used that $\ln(1+\beta) \leq \beta$.
The result follows.
\end{proof}

Unfortunately this naive approach of using the pointwise bound from Lemma \ref{foot} is specific to dimension 1, 
since in higher dimension the tail of the  $\rchi^2$ distribution does not have the correct behavior.
It should be noticed that Ball et al. \cite{BBBOW13} also have a quick direct proof of the Talagrand conjecture for $n=1$ 
that also uses a similar tail comparison with the $\rchi^2$ distribution, and also noticed that such a tail is not of the correct order for $n \geq 2$.

\section{The Deviation Inequality for Log-Convex Functions}
\label{sec:logcvx}

Throughout this section $f \colon \mathbb{R}^n \to \mathbb{R}$ is a \emph{convex} function satisfying $\int e^f \,d\gamma_n =1$ where $\gamma_n$ is the standard Gaussian measure on $\mathbb{R}^n$.
Given $s \in \mathbb{R}$, let 
\[
A_s:=\{f \leq s \}
\] 
and 
\[
\varphi(s):=\Phi^{-1}\left( \gamma_n(A_s) \right),
\]
where 
$\Phi^{-1}$ is the inverse of the Gaussian cumulative function $\Phi$ defined by \eqref{eq:Phi}.

The key ingredient in the proof of Theorem \ref{main-result} is the concavity of the function $\varphi$ that, as we shall see in the proof of the next lemma, is a direct consequence of Ehrhard's inequality \eqref{eq:Ehr83}.

\begin{lem} \label{lem:g}
Let $f$ and $\varphi$ be defined as above. Then $\varphi$ is concave, non-decreasing, $\lim_{s \to \infty} \varphi(s)=+ \infty$ and $\lim_{s \to -\infty} \varphi(s)=- \infty$.
\end{lem}
The concavity of $\varphi$ was first observed by Ehrhard in \cite{Ehr83}. Below we recall the proof for the reader's convenience.

\begin{proof}
That $\varphi$ is non-decreasing and satisfies $\lim_{s \to \infty} \varphi(s)=+ \infty$ and
$\lim_{s \to -\infty} \varphi(s)=- \infty$ is a direct and obvious consequence of the definition.
Now we prove that $\varphi$ is concave, using Ehrhard's inequality. Given $\lambda \in [0,1]$ and $s_1, s_2 \in \mathbb{R}$, we have, by convexity of $f$,
\[
A_{\lambda s_1 + (1-\lambda)s_2} \supset \lambda A_{s_1} + (1-\lambda)A_{s_2}.
\]
Hence, by monotonicity of $\Phi^{-1}$, it holds
\[
\varphi(\lambda s_1 + (1-\lambda)s_2) \geq \Phi^{-1} \left( \gamma_n( \lambda A_{s_1} + (1-\lambda)A_{s_2}) \right) .
\]
Then, Ehrhard's inequality \eqref{eq:Ehr83} implies that
\begin{align*}
\Phi^{-1} \left( \gamma_n( \lambda A_{s_1} + (1-\lambda)A_{s_2}) \right) 
& \geq 
\lambda \Phi^{-1} \left( \gamma_n(  A_{s_1} ) \right) + (1-\lambda) \Phi^{-1} \left( \gamma_n( A_{s_2}) \right) \\
& = 
\lambda \varphi(s_1) + (1-\lambda)\varphi(s_2) 
\end{align*}
from which the concavity of $\varphi$ follows.
\end{proof}

\proof[Proof of Theorem \ref{main-result}]
Let $f$ and $\varphi$ be defined as above. Then, it is enough to show that  
\[
\varphi(u) \geq \sqrt{2u},\qquad \forall u \geq 0.
\]
Since $-\varphi : \R \to \R \cup\{+\infty\}$ is convex by Lemma \ref{lem:g} and lower-semicontinuous,
the Fenchel-Moreau Theorem applies and guarantees that
\[
-\varphi(u)= \sup_{t \in \mathbb{R}} \left\{ ut - \psi(t)\right\},\qquad \forall u \in \R,
\] 
where 
\[
\psi(t)=(-\varphi)^*(t):=\sup_{u \in \mathbb{R}} \left\{ ut + \varphi(u) \right\}
\]
is the Fenchel-Legendre transform of $-\varphi$. Also we observe that, since $\lim_{u \to \infty} \varphi(u)=+\infty$, necessarily $\psi(t)=+\infty$ for all $t >0$ so that 
\[
\varphi(u)= - \sup_{t \leq 0} \left\{ ut - \psi(t)\right\} = \inf_{t \leq 0} \left\{ -ut + \psi(t)\right\} .
\]
Now observe that
\[
1 = \int e^f\, d\gamma_n = \int_{-\infty}^\infty e^u \gamma_n(f \geq u) \,du = \int_{-\infty}^\infty e^u (1-\Phi(\varphi(u))\,du
= \int_{-\infty}^\infty e^u \overline{\Phi}(\varphi(u))\,du
\]
where we recall that $\overline{\Phi}=1-\Phi$.
Using integration by parts and the fact $\overline{\Phi}$ is decreasing, we have for all $t \leq 0$
\begin{align*}
1  = \int_{-\infty}^\infty e^u \overline{\Phi}(\varphi(u))\, du & \geq
\int_{-\infty}^\infty e^u \overline{\Phi}(-ut+\psi(t))\, du\\
& = (-t)e^{\frac{\psi(t)}{t}}\int_{-\infty}^{+\infty}e^{\frac{-v}{t}}\overline{\Phi}(v)\,dv\\
& =  e^{\frac{\psi(t)}{t}} \frac{1}{\sqrt{2\pi}}\int_{-\infty}^{+\infty} e^{\frac{-v}{t}} e^{-v^2/2}\,dv\\
& = \exp\left\{ \frac{\psi(t)}{t} + \frac{1}{2t^2} \right\}.
\end{align*}

Therefore, for all $t \leq 0$ it holds
\[
\psi(t) \geq -\frac{1}{2t} .
\]
In turn,
\[
\varphi(u)= \inf_{t \leq 0} \left\{ -ut + \psi(t)\right\} \geq \inf_{t \leq 0} \left\{ -ut - \frac{1}{2t} \right\} 
 = \sqrt{2u} 
\]
as expected.
\endproof

\section{Revisiting the deviation inequality, with a discussion of the semi-convex case}\label{sec:conj}
\label{sec:logsemicvx}

Suppose that $f : \R^n \to \R$ is a function such that $\int e^f\,d\gamma_n = 1.$
Define $\mu_f$ the distribution of $f$ under $\gamma_n$, that it to say
\[
\mu_f(A) := \gamma_n (\{x \in \R^n : f(x) \in A\}),\qquad \forall \text{ Borel } A \subset \R.
\]
Consider the monotone rearrangement transport map $T_f$ sending $\gamma_1$ onto $\mu_f$. It is defined by
\[
T_f(u) = F_f^{-1}\circ \Phi(u),\qquad \forall u \in \R,
\]
where $F_f(t) = \mu_f((-\infty,t])$, $t \in \R$, denotes the cumulative distribution function of $\mu_f$ and 
\[
F_f^{-1}(s) = \inf\{ t : F_f(t)\geq s\},\qquad s\in (0,1)
\]
its generalized inverse.

The following proposition will yield to a slightly different proof of Theorem \ref{main-result}.
\begin{prop}\label{prop-Tsemi}
With the notation above, if $T_f$ is $\kappa$-semiconvex, for some $\kappa\geq0$ \textit{i.e}
\[
T_f((1-t)x+ty) \leq (1-t)T_f(x) + t T_f(y) + \frac{\kappa}{2} t(1-t) |x-y|^2,\qquad \forall x,y \in \R,\quad \forall t \in [0,1],
\]
then 
\[
\gamma_n(\{f>u\}) \leq \overline{\Phi}\left(\sqrt{2u-\log(1+\kappa)}\right),\qquad \forall u\geq \frac{1}{2}\log(1+\kappa).
\]
\end{prop}
\proof
The $\kappa$-semiconvexity condition is equivalent to the convexity of the function $x\mapsto T_f(x)+\kappa \frac{x^2}{2}$.
Now observe that
\[
1=\int e^f \,d\gamma_n = \int e^y \,d\mu_f(y) = \int e^{T_f(x)}\,d\gamma_1(x).
\]
Applying Lemma \ref{foot} to the function $T_f$ in dimension $1$, one concludes that 
\[
T_f(x) \leq \frac{1}{2}x^2 + \frac{1}{2}\log(1+\kappa),\qquad \forall x \in \R. 
\]
This is equivalent to
\[
\Phi(x) \leq F_f\left(\frac{1}{2}x^2 + \frac{1}{2}\log(1+\kappa)\right)
\]
and thus
\[
F_f(u) \geq \Phi\left(\sqrt{2u-\log(1+\kappa)}\right),\qquad \forall u\geq \frac{1}{2}\log(1+\kappa)
\]
or in other words,
\[
\gamma_n(\{f>u\}) \leq \overline{\Phi}\left(\sqrt{2u-\log(1+\kappa)}\right),\qquad \forall u\geq \frac{1}{2}\log(1+\kappa)
\]
\endproof
\medskip

\proof[Second proof of Theorem \ref{main-result}]
Suppose that $f : \R^n \to \R$ is convex and such that $\int e^f\,d\gamma_n=1$. Then according to Lemma \ref{lem:g}, the function $\Phi^{-1}\circ F_f = T_f^{-1}$ is concave. Being also non-decreasing, its inverse $T_f$ is convex. Applying Proposition \ref{prop-Tsemi} with $\kappa=0$ completes the proof. 
\endproof

In  view of Proposition \ref{prop-Tsemi}, a natural conjecture would be the following:

\medskip

\noindent\textbf{Conjecture.} \textit{There exists a function $\kappa : [0,\infty) \to [0,\infty)$ such that if $f: \R^n \to \R$ is a smooth function such that $\mathrm{Hess}\,f \geq -\beta \mathrm{Id}$, $\beta \geq0$, then the map $T_f$ is $\kappa(\beta)$-semiconvex on $\R$.}

\medskip

If this conjecture was true, then one would recover completely Eldan-Lee-Lehec result \eqref{eq:start}. Besides the convex case, let us observe that the conjecture is obviously true in dimension $1$ for \emph{non-decreasing} functions $f$. Indeed, $f$ is clearly a transport map between $\gamma_1$ and $\mu_f$.  Being non-decreasing, $f$ is necessarily the monotone rearrangement map, that is to say : $f = T_f$. Since $f$ is $\kappa$-semiconvex, then so is $T_f.$

Unfortunately, this probably too naive conjecture turns out to be false in general. As explained to us by R. van Handel, the presence of local minimizers for $f$ breaks down the semi-convexity of $T_f.$
Let us illustrate this in dimension $1$.
Consider a function $f:\R \to \R$ of class $\mathcal{C}^1$ such that $f'(x)$ vanishes only at a finite number of points and such that there is some point $x_o \in \R$ and $\eta >0$ such that $f'(x_o)=0$, $f'(x)<0$ on $[x_o-\eta,x_o[$ and $f'(x_o)>0$ on $]x_o,x_o+\eta].$ Denoting by $t_o = f(x_o)$, we assume that $\inf_\R f < t_o$, that is to say, $f$ only presents a local minimizer at $x_o.$ Let us further assume that there are some $\alpha_o,\beta_o>0$ and some positive integer $N$ such that, for all $t_o-\alpha_o\leq t<t_o$, 
\[
\mathrm{Card} \{ x \in \R : f(x)=t\} \leq N
\]
and $|f'(x)|\geq \beta_o$ for all $x$ such that $t_o-\alpha_o \leq f(x)<t_o$. 
\medskip

\textbf{Claim.} There is no $\lambda \geq 0$ for which the map $T:=T_f$ is $\lambda$-semi-convex. 
\medskip

It is not difficult to exhibit semi-convex functions $f$ enjying the assumptions above, which disclaim the conjecture.

\proof[Proof of the Claim.]
First let us remark that if $T$ was $\lambda$-semi-convex for some $\lambda\geq 0$, then the map $x\mapsto T(x) + \frac{\lambda}{2}x^2$ would be convex, and so would admit finite left and right derivatives everywhere. Moreover for a convex function the left derivative at some point is always less than or equal to the right derivative at this same point. So the $\lambda$-semi-convexity of $T$ would in particular imply that
\[
T_-'(x) \leq T'_+(x),\qquad \forall x \in \R.
\]
We are going to show that $T_-'(u_o) > T'_+(u_o)$ for some $u_o \in \R$ which will prove the claim. Since, denoting $F:=F_f$, 
\[
T'_{\pm}(u) = \frac{\varphi(u)}{F'_\pm \circ T(u)},
\]
at every point $u\in \R$ where the derivative exists, one conludes that it is enough to show that 
\[
F'_{-}(t_o)<F'_{+}(t_o)
\]
to have the desired inequality at $u_o=T^{-1}(t_o)$. Note that $|T^{-1}(t_o)| <\infty$ because $\mu_f((t_o,+\infty))=\gamma_1((T^{-1}(t_o),+\infty))>0$ and $\mu_f((\infty,t_o))=\gamma_1((-\infty,T^{-1}(t_o)))>0$, as easily follows from our assumptions.

According to the one dimensional general change of variable formula, the probability measure $\mu_f$ admits the following density
\[
h(t) = \sum_{x \in \{f=t\}} \frac{\varphi(x)}{|f'(x)|},\qquad  t \in \R, 
\]
where $\varphi(x) = \frac{1}{\sqrt{2\pi}} e^{-x^2/2}$, $x \in \R.$
Define $\varepsilon_o = \max_{[x_o-\eta, x_o+\eta]}f-t_o>0$ ; then, for $h< \varepsilon_o$, it holds

\[
F(t_o+h)-F(t_o) = \int_{t_o}^{t_o+h} h(t)\,dt  \geq h\frac{m}{M(h)},
\]
where 
\[
m = \inf_{[x_o-\eta, x_o+\eta]}\varphi
\]
 and 
\[
M(h) = \sup \left\{ |f'(x)| : x \in [x_o-\eta, x_o+\eta], f(x) \in [t_o,t_o+h]\right\}.
\]
It is easily seen that $M(h)\to 0$ as $h$ to $0^+$, which implies that $F'_+(t_o)=+\infty.$ Now let us consider the left derivative. Let us note that one can assume without loss of generality that the left derivative exists at $t_o$, since otherwise the function $T$ would clearly not be semi-convex. 
For any $h>0$, it holds
\[
F(t_o)-F(t_o-h) = \int_{t_o-h}^{t_o} h(t)\,dt \leq h\frac{N}{\sqrt{2\pi}\beta_o}
\]
and so $F'_-(t_o) < +\infty$, which completes the proof of the claim.
\endproof

\bibliographystyle{plain}

\end{document}